\documentclass{amsart}

\usepackage{amsfonts}
\usepackage[utf8]{inputenc}
\usepackage{amsmath}
\usepackage{lmodern}

\DeclareMathOperator{\Vol}{Vol}
\DeclareMathOperator{\Ric}{Ric}
\DeclareMathOperator{\can}{can}

\newcommand{\s}{\mathbb{S}}
\newcommand{\R}{\mathbb{R}}
\newcommand{\E}{\mathcal{E}}
\newcommand{\M}{\mathcal{M}}
\newcommand{\Y}{\mathcal{Y}}

\newtheorem{theorem}{Theorem}[section]
\newtheorem{proposition}[theorem]{Proposition}
\newtheorem{lemma}[theorem]{Lemma}

\newtheorem{remark}[theorem]{Remark}

\author{Abraão Mendes}
\email{abraao.mendes@im.ufal.br}
\address{Institute of Mathematics, Federal University of Alagoas, Maceió, AL, 57072-970, Brazil}
\date{\today}

\title[Rigidity of volume-minimizing hypersurfaces in 5-manifolds]{Rigidity of volume-minimizing hypersurfaces in Riemannian 5-manifolds}
\thanks{The author is grateful to Fernando C. Marques, Marcos P. Cavalcante, Feliciano Vitório, and Ezequiel Barbosa for their kind interest in this work. The author was financially supported by CAPES Foundation, Ministry of Education of Brazil.}

\begin{document}

\begin{abstract}
In this paper we generalize the main result of \cite{BarrosBatistaCruzSousa} for manifolds that are not necessarily Einstein. In fact, we obtain an upper bound for the volume of a locally volume-minimizing closed hypersurface $\Sigma$ of a Riemannian 5-manifold $M$ with scalar curvature bounded from below by a positive constant in terms of the total traceless Ricci curvature of $\Sigma$. Furthermore, if $\Sigma$ saturates the respective upper bound and $M$ has nonnegative Ricci curvature, then $\Sigma$ is isometric to $\s^4$ up to scaling and $M$ splits in a neighborhood of $\Sigma$. Also, we obtain a rigidity result for the Riemannian cover of $M$ when $\Sigma$ minimizes the volume in its homotopy class and saturates the upper bound.
\end{abstract}

\maketitle

\section{Introduction}

A classical result due to Toponogov \cite{Toponogov} says that the length of any closed simple geodesic $\gamma$ on a closed Riemannian surface $M^2$ satisfies 
\begin{eqnarray*}
L(\gamma)^2\inf_MK\le4\pi^2,
\end{eqnarray*} 
where $K$ is the Gaussian curvature of $M$. Furthermore, if equality holds, then $M^2$ is isometric to the standard unit 2-sphere $\s^2\subset\R^3$ up to scaling (see \cite{HangWang} for a different proof). 

A similar result could be imagined for minimal 2-spheres, instead of closed simple geodesics, in dimension 3. But, it turns out that there is no area bound for minimal 2-spheres in Riemannian 3-manifolds, as pointed out by Marques and Neves \cite{MarquesNeves}. Therefore, an extra hypothesis is needed.

It is well known that if $\Sigma^2$ is a stable minimal 2-sphere in a Riemannian 3-manifold $M^3$, then the area of $\Sigma$ satisfies
\begin{eqnarray}\label{eq.Bray.Brendle.Neves}
A(\Sigma)\left(\frac{\inf_MR}{2}\right)\le4\pi,
\end{eqnarray}
where $R$ is the scalar curvature of $M$. Moreover, if equality holds, then $\Sigma$ is totally geodesic and $R$ is constant equal to $\inf_MR$ on $\Sigma$. If we further assume that $\Sigma$ is locally area-minimizing, then  equality in \eqref{eq.Bray.Brendle.Neves} implies $M$ to be isometric to $(-\varepsilon,\varepsilon)\times\s^2$ up to scaling in a neighborhood of $\Sigma$, suposing that $\Sigma$ is embedded in $M$. This can be seen as a consequence of Bray, Brendle, and Neves' work \cite{BrayBrendleNeves} (see \cite{MicallefMoraru} for an alternative proof).

In dimension $n\ge3$, it is not difficult to construct manifolds $\Sigma^n$ with scalar curvature $R_\Sigma\ge\alpha_n>0$, for some constant $\alpha_n$ depending only on $n$, and arbitrarily large volume. For example, consider $\Sigma_r^n=\s^{n-1}\times S^1(r)$, where $\s^{n-1}\subset\R^n$ is the standard unit $(n-1)$-sphere and $S^1(r)\subset\R^2$ is the circle of radius $r>0$. Clearly, $R_{\Sigma_r}=(n-1)(n-2)$ and $\Vol(\Sigma_r)\longrightarrow\infty$ as $r\longrightarrow\infty$. However, these manifolds are not diffeomorphic to $\s^n$. 

For the spherical case, Gromov and Lawson \cite{GromovLawson} developed a method which permits to construct metrics on $\Sigma^n=S^n$ with scalar curvature $R_\Sigma\ge n(n-1)$ and arbitrarily large volume if $n\ge3$.

These examples show that the analogous inequality to \eqref{eq.Bray.Brendle.Neves} is not true in general in dimension $n\ge3$ even for volume-minimizing hypersurfaces, as we can see taking $M^{n+1}=\R\times\Sigma^n$ with $\Sigma^n$ as above.

Bearing this in mind, Barros, Batista, Cruz, and Sousa \cite{BarrosBatistaCruzSousa} considered the case of Einstein 4-manifolds embedded in Riemannian 5-manifolds which minimize the volume in their homotopy classes. They proved:

\begin{theorem}[Barros-Batista-Cruz-Sousa]\label{theorem.BarrosBatistaCruzSousa}
Let $M^5$ be a complete Riemannian manifold with positive scalar curvature and nonnegative Ricci curvature. Suppose that $\Sigma^4$ is a two-sided closed Einstein manifold embedded in $M^5$ in such a way that $\Sigma$ minimizes the volume in its homotopy class. Then, the volume of $\Sigma$ satisfies
\begin{eqnarray}\label{eq.BarrosBatistaCruzSousa}
\Vol(\Sigma)^{1/2}\left(\frac{\inf_MR}{12}\right)\le\Vol(\s^4)^{1/2}.
\end{eqnarray}
Moreover, if equality holds, then $\Sigma$ is isometric to $\s^4$, $M$ is isometric to $(-\varepsilon,\varepsilon)\times\s^4$ in a neighborhood of $\Sigma$, and the Riemannian cover of $M$ is isometric to $\R\times\s^4$, up to scaling.
\end{theorem} 

Our purpose in this work is to generalize Theorem \ref{theorem.BarrosBatistaCruzSousa} for manifolds that are not necessarily Einstein. To do so, from the above comments, it is necessary an extra term in \eqref{eq.BarrosBatistaCruzSousa}. Our first result is the following:

\begin{theorem}[Theorem \ref{main.theorem}]\label{main.theorem.introduction}
Let $M^5$ be a Riemannian manifold with scalar curvature $R$ satisfying $\inf_MR>0$ and nonnegative Ricci curvature. If $\Sigma^4$ is a two-sided closed hypersurface embedded in $M^5$ which is locally volume-minimizing, then the volume of $\Sigma$ satisfies
\begin{eqnarray}\label{eq.intro.main.theorem.1}
\Vol(\Sigma)\left(\frac{\inf_MR}{12}\right)^2\le\Vol(\s^4)+\frac{1}{12}\int_\Sigma|\mathring{\Ric_\Sigma}|^2d\sigma,
\end{eqnarray}
where $\mathring{\Ric_\Sigma}$ is the traceless Ricci tensor of $\Sigma$. Furthermore, if equality holds, then $\Sigma$ is isometric to $\s^4$ and $M$ is isometric to $(-\varepsilon,\varepsilon)\times\s^4$ in a neighborhood of $\Sigma$, up to scaling.
\end{theorem}

Our second result is the following:

\begin{theorem}[Theorem \ref{main.theorem.2}]\label{main.theorem.introduction.2}
Let $M^5$ be a complete Riemannian manifold with scalar curvature $R$ satisfying $\inf_MR>0$ and nonnegative Ricci curvature. Suppose that $\Sigma^4$ is a two-sided closed manifold immersed in $M^5$ in such a way that $\Sigma$ minimizes the volume in its homotopy class. Then, the volume of $\Sigma$ satisfies
\begin{eqnarray}\label{eq.intro.main.theorem.2}
\Vol(\Sigma)\left(\frac{\inf_MR}{12}\right)^2\le\Vol(\s^4)+\frac{1}{12}\int_\Sigma|\mathring{\Ric_\Sigma}|^2d\sigma.
\end{eqnarray}
Moreover, if equality holds, then $\Sigma$ is isometric to $\s^4$ and the Riemannian cover of $M$ is isometric to $\R\times\s^4$, up to scaling.
\end{theorem}

\begin{remark}
{\em The covering map of Theorem \ref{main.theorem.2} is explicit. In fact, it is given by $G(t,x)=\exp_x(tN(x))$, $(t,x)\in\R\times\Sigma$, where $\exp$ is the exponential map of $M$ and $N$ is a unit normal vector field defied on $\Sigma$.}
\end{remark}

\section{Preliminaries}

In this section, we are going to present some terminologies and useful results.

Let $\Sigma^n$ be a connected closed (compact without boundary) manifold of dimension $n\ge3$. Denote by $\mathcal{M}(\Sigma)$ the set of all Riemannian metrics on $\Sigma$. The {\em Einstein-Hilbert functional} $\mathcal{E}:\mathcal{M}(\Sigma)\to\R$ is defined by 
\begin{eqnarray*}
 \E(g)=\dfrac{\int_\Sigma R_gdv_g}{\Vol(\Sigma^n,g)^{\frac{n-2}{n}}},
\end{eqnarray*} 
where $R_g$ is the scalar curvature of $(\Sigma,g)$. Denote by $[g]=\{e^{2f}g:f\in C^\infty(\Sigma)\}$ the conformal class of $g\in\M(\Sigma)$. The {\em Yamabe invariant} of $(\Sigma,[g])$ is defined as the following conformal invariant:
\begin{eqnarray*}
 \Y(\Sigma,[g])=\inf_{\tilde g\in[g]}\E(\tilde g).
\end{eqnarray*}

The classical solution of the Yamabe problem by Yamabe \cite{Yamabe}, Trudinger \cite{Trudinger}, Aubin \cite{Aubin76} (se also \cite{Aubin98}), and Schoen \cite{Schoen} says that every conformal class $[g]$ contains metrics $\hat g$, called {\em Yamabe metrics}, which realize the minimum:
\begin{eqnarray*}
 \E(\hat g)=\Y(\Sigma,[g]).
\end{eqnarray*} 
Such metrics have constant scalar curvature given by 
\begin{eqnarray*}
 R_{\hat g}=\Y(\Sigma^n,[g])\Vol(\Sigma^n,\hat g)^{-\frac{2}{n}}. 
\end{eqnarray*}
Furthermore,
\begin{eqnarray*}
 \Y(\Sigma^n,[g])\le\Y(\s^n,[g_{\can}])
\end{eqnarray*}
and equality holds if and only if $(\Sigma^n,g)$ is conformally diffeomorphic to the standard unit $n$-sphere $\s^n\subset\R^{n+1}$ endued with the canonical metric $g_{\can}$. Therefore, as a consequence of Obata's theorem \cite[Proposition 6.1]{Obata}, if $\Y(\Sigma^n,[g])=\Y(\s^n,[g_{\can}])$ and $g$ has constant scalar curvature, then $(\Sigma^n,g)$ is isometric to $(\s^n,g_{\can})$ up to scaling.

When $n=4$, a very useful tool is the Gauss-Bonnet-Chern formula for the Euler characteristic $\chi(\Sigma)$ of a closed Riemannian manifold $(\Sigma^4,g)$, which reads as follows:
\begin{eqnarray*}
8\pi^2\chi(\Sigma)=\int_\Sigma\left(\frac{1}{4}|W_g|^2+\frac{1}{24}R_g^2-\frac{1}{2}|\mathring{\Ric_g}|^2\right)d\sigma,
\end{eqnarray*}
where $W_g$ and $\mathring{\Ric_g}=\Ric_g-(R_g/n)g$ are the Weyl and the traceless Ricci tensors of $(\Sigma,g)$, respectively.

Before finishing this section, we are going to state two important inequalities proved by Gursky \cite{Gursky}.

\begin{theorem}[Gursky]
Let $(\Sigma^4,g)$ be a closed Riemannian manifold. If $\Sigma$ has nonnegative scalar curvature, then 
\begin{eqnarray}\label{gursky.aux.1.1}
\int_\Sigma|W_g|^2d\sigma\ge32\pi^2(\chi(\Sigma)-2)
\end{eqnarray}
and
\begin{eqnarray}\label{gursky.aux.1.2}
\Y(\Sigma,[g])^2\ge6\left(32\pi^2\chi(\Sigma)-\int_\Sigma|W_g|^2d\sigma\right).
\end{eqnarray}
\end{theorem}

\begin{remark}
{\em Clearly, \eqref{gursky.aux.1.1} and \eqref{gursky.aux.1.2} are trivial if $\chi(\Sigma)\le 2$ or $\chi(\Sigma)\le0$, respectively.}
\end{remark}

\section{The results}

Let $\Sigma^4$ be a closed hypersurface immersed in a Riemannian manifold $M^5$. Here, we suppose that $\Sigma$ is {\em two-sided}, that is, there exists a unit normal vector field $N$ defined on $\Sigma$.

\begin{proposition}\label{main.proposition}
Let $M^5$ be a Riemannian manifold with scalar curvature $R$ satisfying $\inf_MR>0$ and $\Sigma^4$ be a two-sided closed hypersurface immersed in $M^5$. If $\Sigma$ is stable minimal in $M$, then the volume of $\Sigma$ satisfies
\begin{eqnarray}\label{eq.main.prop}
\Vol(\Sigma)\left(\frac{\inf_MR}{12}\right)^2\le\Vol(\s^4)+\frac{1}{12}\int_\Sigma|\mathring{\Ric_\Sigma}|^2d\sigma,
\end{eqnarray}
where $\mathring{\Ric_\Sigma}$ is the traceless Ricci tensor of $\Sigma$. Furthermore, if equality holds, then
\begin{enumerate}
\item[\rm (i)] $\Sigma$ is isometric to $\s^4$ up to scaling,
\item[\rm (ii)] $\Sigma$ is totally geodesic in $M$,
\item[\rm (iii)] $\Ric(N,N)=0$ and $R=\inf_MR$ on $\Sigma$,
\end{enumerate}
where $\Ric$ is the Ricci tensor of $M$.
\end{proposition}

\begin{proof}
Since the left hand side of \eqref{eq.main.prop} and $\int_\Sigma|\mathring{\Ric_\Sigma}|^2d\sigma$ are scaling invariant and $\inf_MR>0$, without loss of generality, we may assume that $\inf_MR=12$. Since $\Sigma$ is stable minimal, the stability inequality says that
\begin{eqnarray}\label{main.prop.aux.1}
\int_\Sigma(|\nabla f|^2-(\Ric(N,N)+|A|^2)f^2)d\sigma\ge0
\end{eqnarray} 
for all $f\in C^\infty(\Sigma)$, where $A$ is the second fundamental form of $\Sigma$ in $M$. Taking $f=1$ above and using the Gauss equation
\begin{eqnarray}\label{main.prop.aux.2}
\frac{1}{2}(R-R_\Sigma+|A|^2)=\Ric(N,N)+|A|^2,
\end{eqnarray}
we have
\begin{eqnarray}\label{main.prop.aux.3}
\int_\Sigma(R+|A|^2)d\sigma\le\int_\Sigma R_\Sigma d\sigma,
\end{eqnarray}
where $R_\Sigma$ is the scalar curvature of $\Sigma$. Therefore, observing that $R\ge12$ by hypothesis, it follows that
\begin{eqnarray}\label{main.prop.aux.4}
12\Vol(\Sigma)\le\int_\Sigma R_\Sigma d\sigma\le\Vol(\Sigma)^{1/2}\left(\int_\Sigma R_\Sigma^2d\sigma\right)^{1/2},
\end{eqnarray}
i.e., 
\begin{eqnarray}\label{main.prop.aux.5}
12^2\Vol(\Sigma)\le\int_\Sigma R_\Sigma^2d\sigma,
\end{eqnarray}
where above we have used the H\"older inequality.

Now, let $\phi\in C^\infty(\Sigma)$, $\phi>0$, be the first eigenfunction of the stability operator $L$ of $\Sigma$, 
\begin{eqnarray}\label{main.prop.aux.6}
L=\Delta+(\Ric(N,N)+|A|^2),
\end{eqnarray}
associated to the first eigenvalue $\lambda=\lambda_1$, that is, 
\begin{eqnarray}\label{main.prop.aux.7}
L\phi+\lambda\phi=0.
\end{eqnarray}
Because $\Sigma$ is stable, $\lambda\ge0$. Denote by $g_\Sigma$ the Riemannian metric on $\Sigma$ induced from $M$ and define a new metric $g=\phi^{2/3}g_\Sigma$. It is well known that the scalar curvatures of $g$ and $g_\Sigma$ are related according to the equations
\begin{eqnarray*}
R_g\phi&=&-6\Delta(\phi^{1/3})+R_\Sigma\phi^{1/3}\\
&=&-2\phi^{-2/3}\Delta\phi+\frac{4}{3}\phi^{-5/3}|\nabla\phi|^2+R_\Sigma\phi^{1/3},
\end{eqnarray*}
which imply
\begin{eqnarray}\label{main.prop.aux.9}
R_g\phi\ge-2\phi^{-2/3}\Delta\phi+R_\Sigma\phi^{1/3}.
\end{eqnarray}
Then, using \eqref{main.prop.aux.2}, \eqref{main.prop.aux.6} and \eqref{main.prop.aux.7} into \eqref{main.prop.aux.9}, we obtain
\begin{eqnarray*}
R_g\phi&\ge&\phi^{1/3}(R+|A|^2+2\lambda)\ge12\phi^{1/3}>0,
\end{eqnarray*}
thus $R_g>0$. In particular, $\Y(\Sigma,[g_\Sigma])>0$. Denoting by $W_\Sigma$ the Weyl tensor of $(\Sigma,g_\Sigma)$ and remembering that $\int_\Sigma|W_\Sigma|^2d\sigma$ is a conformal invariant of $(\Sigma,g_\Sigma)$ in dimension 4, it follows from Gursky's theorem that 
\begin{eqnarray}\label{main.prop.aux.11}
\int_\Sigma|W_\Sigma|^2d\sigma\ge32\pi^2(\chi(\Sigma)-2).
\end{eqnarray}
Then, using \eqref{main.prop.aux.5}, \eqref{main.prop.aux.11} and the Gauss-Bonnet-Chern formula, we have
\begin{eqnarray*}
\Vol(\Sigma)&\le&\frac{1}{6}\left(\frac{1}{24}\int_\Sigma R_\Sigma^2d\sigma\right)\\
&=&\frac{1}{6}\left(8\pi^2\chi(\Sigma)-\frac{1}{4}\int_\Sigma|W_\Sigma|^2d\sigma+\frac{1}{2}\int_\Sigma|\mathring{\Ric_\Sigma}|^2d\sigma\right)\\
&\le&\frac{8}{3}\pi^2+\frac{1}{12}\int_\Sigma|\mathring{\Ric_\Sigma}|^2d\sigma,
\end{eqnarray*}
which imply \eqref{eq.main.prop} because $\Vol(\s^4)=\frac{8}{3}\pi^2$.

If equality in \eqref{eq.main.prop} holds, then we have equality in \eqref{main.prop.aux.1} for $f=1$. Which means that $\lambda_1=0$ and $f=1$ is the first eigenfunction of $L$, i.e., $\Ric(N,N)+|A|^2=0$. On the other hand, equality in \eqref{eq.main.prop} also implies equality in \eqref{main.prop.aux.3} and \eqref{main.prop.aux.4}. Therefore, 
\begin{eqnarray*}
0&\le&\int_\Sigma(R-12)d\sigma\le\int_\Sigma(R+|A|^2)d\sigma-12\Vol(\Sigma)\\
&=&\int_\Sigma R_\Sigma d\sigma-12\Vol(\Sigma)=0,
\end{eqnarray*}
thus $\Sigma$ is totally geodesic and $R=12$ on $\Sigma$. In particular, $\Ric(N,N)=0$ on $\Sigma$. Also, from \eqref{gursky.aux.1.2} we have $R_\Sigma=12$.

To finish, observe that equality in \eqref{eq.main.prop} implies equality in \eqref{main.prop.aux.11}. Therefore, since $\Y(\Sigma,[g_\Sigma])>0$, using Gursky's theorem we obtain
\begin{eqnarray*}
\displaystyle\Y(\s^4,[g_{\can}])^2&\ge&\Y(\Sigma,[g_\Sigma])^2\ge6\left(32\pi^2\chi(\Sigma)-\int_\Sigma|W_\Sigma|^2d\sigma\right)\\
&=&384\pi^2=\Y(\s^4,[g_{\can}])^2.
\end{eqnarray*}
Then, $\Y(\Sigma,[g_\Sigma])=\Y(\s^4,[g_{\can}])$ and $R_\Sigma=12$, which imply by the solution of the Yamabe problem and Obata's theorem that $\Sigma$ is isometric to $\s^4$.
\end{proof}

\begin{remark}
{\em It follows from the above proposition that if equality in \eqref{eq.intro.main.theorem.1} or \eqref{eq.intro.main.theorem.2} holds, then $\Sigma$ is isometric to $\s^4$ up to scaling. In particular, $\Sigma$ is Einstein. In this case, we can use Barros-Batista-Cruz-Sousa's theorem to obtain Theorem \ref{main.theorem.introduction} and Theorem \ref{main.theorem.introduction.2}. But, for the sake of completeness, we are going to present the proofs of these theorems here.}
\end{remark}

Before proving our main results, we are going to state a very useful lemma due to Bray, Brendle, and Neves \cite{BrayBrendleNeves} (see \cite{Nunes} for a more detailed proof). The same technique has been used by many authors in the literature (e.g. \cite{AnderssonCaiGalloway,BarrosBatistaCruzSousa,Cai,MicallefMoraru}).

\begin{lemma}\label{lemma.foliation}
Let $M$ be a Riemannian $5$-manifold. If $\Sigma$ is a two-sided closed minimal hypersurface immersed in $M$ such that $\Ric(N,N)+|A|^2=0$ on $\Sigma$, then there exists a smooth function $w:(-\varepsilon,\varepsilon)\times\Sigma\to\R$, for some $\varepsilon>0$, satisfying
\begin{eqnarray*}
w(0,x)=0,\,\,\frac{\partial w}{\partial t}(0,x)=1\,\,\mbox{ and }\,\,\int_\Sigma(w(t,\cdot)-t)d\sigma=0
\end{eqnarray*}
for all $x\in\Sigma$ and $t\in(-\varepsilon,\varepsilon)$. Furthermore,  
\begin{eqnarray*}
\Sigma_t=\{\exp_x(w(t,x)N(x))\in M:x\in\Sigma\}
\end{eqnarray*}
is a closed hypersurface immersed in $M$ with constant mean curvature for each $t\in(-\varepsilon,\varepsilon)$. Also, if $\Sigma$ is embedded in $M$, then $\{\Sigma_t\}_{t\in(-\varepsilon,\varepsilon)}$ is a foliation of a neighborhood of $\Sigma=\Sigma_0$.
\end{lemma}

All entities associated to $\Sigma_t$ will be denoted with a subscript $t$, except the mean curvature which will be denoted by $H(t)$. Furthermore, $\rho_t$ will denote the lapse function $\langle\frac{\partial}{\partial t},N_t\rangle$.

\begin{theorem}\label{main.theorem}
Let $M^5$ be a Riemannian manifold with scalar curvature $R$ satisfying $\inf_MR>0$ and nonnegative Ricci curvature. If $\Sigma^4$ is a two-sided closed hypersurface embedded in $M^5$ which is locally volume-minimizing, then the volume of $\Sigma$ satisfies
\begin{eqnarray}\label{eq.main.theorem}
\Vol(\Sigma)\left(\frac{\inf_MR}{12}\right)^2\le\Vol(\s^4)+\frac{1}{12}\int_\Sigma|\mathring{\Ric_\Sigma}|^2d\sigma.
\end{eqnarray}
Furthermore, if equality holds, then $\Sigma$ is isometric to $\s^4$ and $M$ is isometric to $(-\varepsilon,\varepsilon)\times\s^4$ in a neighborhood of $\Sigma$, up to scaling.
\end{theorem}

\begin{proof}
Inequality \eqref{eq.main.theorem} follows immediately from Proposition \ref{main.proposition}, since all locally volume-minimizing hypersurfaces are stable minimal. Also, if equality in \eqref{eq.main.theorem} holds, then $\Sigma$ is isometric to $\s^4$ up to scaling and $\Ric(N,N)=0=|A|^2$ on $\Sigma$. In particular, we can use Lemma \ref{lemma.foliation}. It is well known that 
\begin{eqnarray*}
\frac{dH}{dt}(t)=-\Delta_t\rho_t-(\Ric(N_t,N_t)+|A_t|^2)\rho_t.
\end{eqnarray*}
Since $\rho_0=1$ and $\Sigma$ is compact, we may assume that $\rho_t>0$ for all $t\in(-\varepsilon,\varepsilon)$. Therefore, using that $\Ric\ge0$ and $H'(t)$ is constant on $\Sigma_t$, we have 
\begin{eqnarray}
H'(t)\int_{\Sigma_t}\frac{1}{\rho_t}d\sigma_t&=&-\int_{\Sigma_t}\frac{\Delta_t\rho_t}{\rho_t}d\sigma_t-\int_{\Sigma_t}(\Ric(N_t,N_t)+|A_t|^2)d\sigma_t\label{main.theorem.aux.2}\\
&\le&-\int_{\Sigma_t}\frac{|\nabla_t\rho_t|^2}{\rho_t^2}d\sigma_t\le0,\label{main.theorem.aux.3}
\end{eqnarray}
which imply $H'(t)\le0$ for all $t\in(-\varepsilon,\varepsilon)$, and then 
\begin{eqnarray}\label{main.theorem.aux.4}
H(t)\le0\le H(-t)\,\,\mbox{ for all }\,\,t\in[0,\varepsilon),
\end{eqnarray}
because $H(0)=0$. On the other hand, the first variation formula says that 
\begin{eqnarray}\label{main.theorem.aux.5}
\frac{d}{dt}\Vol(\Sigma_t)=\int_{\Sigma_t}H(t)\rho_td\sigma_t.
\end{eqnarray}
Then, \eqref{main.theorem.aux.4} and \eqref{main.theorem.aux.5} imply 
\begin{eqnarray*}
\Vol(\Sigma_t)\le\Vol(\Sigma_0)\,\,\mbox{ for all }\,\,t\in(-\varepsilon,\varepsilon).
\end{eqnarray*}
But, since $\Sigma=\Sigma_0$ is locally volume-minimizing, we have $\Vol(\Sigma_t)=\Vol(\Sigma)$ for all $t\in(-\varepsilon,\varepsilon)$, for a smaller $\varepsilon>0$ if necessary. Therefore, 
\begin{eqnarray*}
0=\frac{d}{dt}\Vol(\Sigma_t)=\int_{\Sigma_t}H(t)\rho_td\sigma_t
\end{eqnarray*} 
and \eqref{main.theorem.aux.4} imply $H(t)=0$ for all $t\in(-\varepsilon,\varepsilon)$. Using $H'(t)=0$ into \eqref{main.theorem.aux.2} and \eqref{main.theorem.aux.3}, we conclude that $\rho_t$ is constant on $\Sigma_t$ and $\Sigma_t$ is totally geodesic in $M$ for each $t\in(-\varepsilon,\varepsilon)$.

Now, we want to prove that $t\longmapsto N_t(x)$ is a parallel vector field along to the curve $t\longmapsto G(t,x)=\exp_x(w(t,x)N(x))$ for each $x\in\Sigma$. In fact, choosing a local coordinate system $x=(x_1,x_2,x_3,x_4)$ on $\Sigma$, we have
\begin{eqnarray*}
\left\langle\nabla_\frac{\partial G}{\partial t}N_t,\frac{\partial G}{\partial x_i}\right\rangle&=&-\left\langle N_t,
\frac{\partial^2 G}{\partial x_i\partial t}\right\rangle=-\frac{\partial}{\partial x_i}\left\langle N_t,\frac{\partial G}{\partial t}\right\rangle+
\left\langle\nabla_\frac{\partial G}{\partial x_i}N_t,\frac{\partial G}{\partial t}\right\rangle\\
&=&-\frac{\partial}{\partial x_i}\rho_t=0.
\end{eqnarray*} 
Above we have used that $\nabla_{\frac{\partial G}{\partial x_i}}N_t=0$ since $\Sigma_t$ is totally geodesic.
Also,
\begin{eqnarray*}
\left\langle\nabla_\frac{\partial G}{\partial t}N_t,N_t\right\rangle=\frac{1}{2}\frac{\partial}{\partial t}\langle N_t,N_t\rangle=0.
\end{eqnarray*} 
Thus, $N_t$ is parallel.

On the other hand, we know that $t\longmapsto d(\exp_x)_{w(t,x)N(x)}N(x)$ is also parallel along to $t\longmapsto G(t,x)$. Then, $N_t(x)=(d\exp_x)_{w(t,x)N(x)}N(x)$ by uniqueness of parallel vector fields, since $w(0,x)=0$ and $N_0(x)=N(x)$. In particular, 
\begin{eqnarray*}
\rho_t=\left\langle\frac{\partial G}{\partial t},N_t\right\rangle=\frac{\partial w}{\partial t}.
\end{eqnarray*}
Now, because $\int_\Sigma(w(t,\cdot)-t)d\sigma=0$ and $\rho_t=\frac{\partial w}{\partial t}$ is constant on $\Sigma_t$, we obtain
\begin{eqnarray*}
0=\frac{d}{dt}\int_\Sigma(w(t,\cdot)-t)d\sigma=\int_\Sigma\left(\frac{\partial w}{\partial t}-1\right)d\sigma=\left(\frac{\partial w}{\partial t}
-1\right)\Vol(\Sigma), 
\end{eqnarray*} 
which imply $\frac{\partial w}{\partial t}(t,x)=1$ for all $(t,x)\in(-\varepsilon,\varepsilon)\times\Sigma$. Finally, because $\frac{\partial w}{\partial t}(0,x)=1$, we have $w(t,x)=t$ for all $(t,x)\in(-\varepsilon,\varepsilon)\times\Sigma$. Therefore, $G(t,x)=\exp_x(tN(x))$ and we can easily check that $G$ is an isometry from $(-\varepsilon,\varepsilon)\times\Sigma$ to a neighborhood of $\Sigma$ in $M$.
\end{proof}

\begin{remark}\label{remark}
{\em Supposing that $\Sigma$ is immersed instead of embedded into $M$ in the above theorem, it follows from the same proof that $\Sigma$ is isometric to $\s^4$ up to scaling and $G$ is a local isometry from $(-\varepsilon,\varepsilon)\times\Sigma$ to $M$, if equality in \eqref{eq.main.theorem} holds.}
\end{remark}

The proof presented below is essentially the same as in \cite{BarrosBatistaCruzSousa}, \cite{BrayBrendleNeves}, and \cite{Nunes}.

\begin{theorem}\label{main.theorem.2}
Let $M^5$ be a complete Riemannian manifold with scalar curvature $R$ satisfying $\inf_MR>0$ and nonnegative Ricci curvature. Suppose that $\Sigma^4$ is a two-sided closed manifold embedded in $M^5$ in such a way that $\Sigma$ minimizes the volume in its homotopy class. Then, the volume of $\Sigma$ satisfies
\begin{eqnarray}\label{eq.main.theorem.2}
\Vol(\Sigma)\left(\frac{\inf_MR}{12}\right)^2\le\Vol(\s^4)+\frac{1}{12}\int_\Sigma|\mathring{\Ric_\Sigma}|^2d\sigma,
\end{eqnarray}
Moreover, if equality holds, then $\Sigma$ is isometric to $\s^4$ and the Riemannian cover of $M$ is isometric to $\R\times\s^4$, up to scaling.
\end{theorem}

\begin{proof}
Inequality \eqref{eq.main.theorem.2} follows directly from Theorem \ref{main.theorem}. Suppose that equality in \eqref{eq.main.theorem.2} holds and define $G:\R\times\Sigma\to M$ by $G(t,x)=\exp_x(tN(x))$. We claim that $G$ is a local isometry. In fact, define $I=\{t>0:G|_{(0,t)\times\Sigma}\,\,\mbox{ is a local isometry}\}$ and observe that Remark \ref{remark} implies $\Sigma$ to be isometric to $\s^4$ up to scaling and $I\neq\emptyset$. In particular, $\mathring{\Ric_\Sigma}=0$. On the other hand, it is not difficult to see that $I$ is closed in $(0,\infty)$. In fact, suppose that $t_k\in I$ converges to $t\in(0,\infty)$. If $t\le t_k$ for some $k$ then $t\in I$ because $(0,t)\times\Sigma\subset(0,t_k)\times\Sigma$ and $G|_{(0,t_k)\times\Sigma}$ is a local isometry. Otherwise, if $t_k<t$ for all $k$ then $\bigcup_k(0,t_k)\times\Sigma=(0,t)\times\Sigma$ implies that $t\in I$ because $G|_{(0,t_k)\times\Sigma}$ is a local isometry (which is a local property) for each $k$. Let us prove that $I$ is also open. Given $t\in I$, we have that $\Sigma_t=G(t,\Sigma)$ is homotopic to $\Sigma$ in $M$, $\Vol(\Sigma_t)=\Vol(\Sigma)$, and $\mathring{\Ric_{\Sigma_t}}=0$, because $G:\{t\}\times\Sigma\to\Sigma_t$ is a local isometric. In particular, $\Sigma_t$ minimizes the volume in its homotopy class and attains the equality in \eqref{eq.main.theorem.2}. Therefore, it follows from Remark \ref{remark} that there exists $\varepsilon>0$ such that $G|_{(0,t+\varepsilon)\times\Sigma}$ is a local isometry. This proves that $I$ is open in $(0,\infty)$. Thus, $I=(0,\infty)$, i.e., $G|_{(0,\infty)\times\Sigma}$ is a local isometry. Analogously, we can prove that $G|_{(-\infty,0)\times\Sigma}$ is a local isometry. This, together with Remark \ref{remark}, implies that $G$ is a local isometry. In particular, $G$ is a covering map.
\end{proof}


\bibliographystyle{amsplain}
\bibliography{bibliography}

\end{document}